\documentclass{article}
\usepackage{graphicx} 
\usepackage{amsmath,amssymb,amsthm,mathtools}
\usepackage{hyperref}
\usepackage{pinlabel}
\usepackage{cite}
\usepackage{xcolor}
\usepackage{comment}
\newtheorem{Theorem}{Theorem}
\newtheorem{Lemma}[Theorem]{Lemma}
\newtheorem{Definition}[Theorem]{Definition}

\numberwithin{Theorem}{section}
\usepackage{sectsty}
\usepackage[margin=1.5in]{geometry}
\sectionfont{\fontsize{12}{15}\selectfont}

\newcommand{\Tsphere}{S^3}
\newcommand{\ST}{V} 

\newcommand{\Slink}{L_n} 
\newcommand{\Stml}{W} 
\newcommand{\STmLone}{X} 
\newcommand{\STmLtwo}{Y} 
\newcommand{\STmLonetwo}{Z} 
\newcommand{\JSJt}{\mathbf{T}} 

\title{Non-characterizing slopes for more $3$-manifolds}
\author{Matthew Elpers}
\begin{document}
\date{}
\maketitle
\begin{abstract}
    For any homotopy class $h$ in any compact orientable $3$-manifold $M$ which is closed or has exclusively torus boundary components, we produce infinitely many pairs of distinct knots representing $h$ with orientation-preserving homeomorphic $0$-surgeries. 
\end{abstract}
\section{Introduction}
\par The topology of $3$-manifolds is deeply connected to knot theory by the well-known result of Lickorish \cite{LW} and Wallace \cite{Wall} which states that every closed and orientable $3$-manifold can be obtained by surgery of a framed link in the $3$-sphere $\Tsphere$. The link used in this surgery is not unique.  Kirby \cite{Kirby} produced an equivalence relation on framed links such that two links are equivalent if and only if they yield the same manifold after surgery. Kirby then asked if two different knots with the same framing could appear in a single equivalence class. Lickorish \cite{Lick} first answered the question in the affirmative, producing two knots with framing $-1$ which share the same surgery. One can give a uniqueness result for these equivalence classes in the language of characterizing slopes. For a knot $K$ in $S^3$ we say $p/q$ is a characterizing slope if whenever $S^3_{p/q}(K)$ is orientation-preserving diffeomorphic to $S^3_{p/q}(K')$, then $K'$ is isotopic to $K$. In the positive direction, Kronheimer, Mrowka, Ozsv\'ath, and Szab\'o \cite{KMOS} showed all non-trivial slopes are characterizing for the unknot. Furthermore, Ozsv\'ath, and Szab\'o \cite{OS} showed the same holds for the trefoil and figure-eight knots. Lackenby \cite{LackIM} showed that every knot has infinitely many characterizing slopes. This, along with work of McCoy \cite{McCoy}, and Sorya \cite{Sorya} proves the existence of a bound $C(K)$ for which all $q>C(K)$ will be characterizing. Later, Sorya and Wakelin \cite{WS} provided a method for computing $C(K)$. The most recent results in characterizing slopes include \cite{HPW} and \cite{KS}.  In the negative direction Baker and Motegi \cite{BM} provided examples of knots admitting infinitely many non-characterizing slopes. 
\par Furthermore, one can try to generalize characterizing slopes to other $3$-manifolds. Although, there is an added difficulty because in arbitrary $3$-manifolds there  non-nullhomologous knots and as a result do not have a Seifert framing. Famous results of Berge \cite{Berge} and Gabai \cite{Gabai} classify all knots in $D^2 \times S^1$ which can produce $D^2 \times S^1$. In $S^2 \times S^1$ M. Kim and J. Park \cite{KP} as well as Hayden, Mark, and Piccirillo \cite{HMP} showed that there are infinitely many pairs of distinct knots in $S^2 \times S^1$ whose surgeries produce the same homology sphere. We prove an analogue for a more general class of $3$-manifolds. Our choice of 0-framing is outlined in Definition \ref{SatFra} and the paragraph which follows it.

\begin{Theorem} \label{mainthm}
    Let $M$ be any compact, connected, orientable $3$-manifold, such that $M$ is closed or $\partial M$ contains only tori. For every homotopy class $h$ in $\pi_1(M)$ there are infinitely many pairs of knots representing $h$ which have orientation-preserving homeomorphic $0$-surgeries. 
\end{Theorem} 
 Now we describe the general method constructing our knots. For any such $M$ and $h$, by a theorem of Myers \cite{Myers}, there is a hyperbolic knot $K$ such that $[K]=h$. We then modify the choice of representative $K$ using the generalized satellite operation from Section $2$.  This modification is performed using a family of pairs $A_n$ and $B_n$ of pattern knots which are defined in Section $3$. Wherein, we verify their key properties; namely: they are hyperbolic, are homotopic to core curves, have exteriors with different hyperbolic volumes, and orientation-preserving homeomorphic $0$-surgeries. Consequently, the satellite knots $A_n(K_\ell)$ and $B_n(K_\ell)$ are homotopic to $K$ and have JSJ decomposition with only two pieces: the original knot exterior and the pattern space. Using this we are able to show $A_n(K_\ell)$ and $B_n(K_\ell)$ satisfy the statement of the theorem.  
\\
\textbf{Organization:} In Section 2 we specify framings for our knots, define a satellite operation, review Dehn surgery, and recall the JSJ decomposition for $3$-manifolds. In Section 3, we define the knots $A_n$ and $B_n$ in the solid torus and verify their key geometric properties. In Section 4, we prove the main theorem.
\\
\textbf{Acknowledgments:} The author would like to thank his advisor Tye Lidman for his insight, patience, and assistance while completing this work. Additionally, thanks to Ken Baker, Allison Moore for helpful conversations regarding this work. As well as Laura Wakelin and Patricia Sorya for reviewing previous versions of this article. ME was partially supported by NSF Grants DMS-2105469 and DMS-2506277. 
\section{Preliminaries} 
     For the remainder of this paper let $M$ be a compact orientable $3$-manifold with exclusively torus boundary components and $V$ denote $D^2 \times S^1$ where $D^2$ is a disk of radius $1$ centered at the origin in $\mathbb{R}^2$ and $S^1$ is the unit interval $[0,1]$ with endpoints identified. The \emph{core curve} $c$ of $V$ is $\{(0,0)\} \times S^1$, the longitude $l$ of $V$ is the curve $\{(1,0)\} \times S^1$, and the meridian $m$ is the curve $\partial D^2 \times \{0\}$.  A \emph{framed knot} $(K,\ell)$ in $M$ is knot $K \subset M$ together with a trivialization $\ell$ of $\nu(K)$, the normal bundle of $K$. Alternatively $\ell$ can be thought of as a parallel copy of $K$ which lies on $\partial\nu(K)$, the boundary of a tubular neighborhood of $K$. As we will see, framed knots can be used in Dehn surgery to construct and study new and interesting examples of $3$-manifolds. 
    \par A \emph{Dehn surgery} on $K \subset M$ involves two steps. First, we remove a tubular neighborhood of $K$ from $M$ to form the knot exterior $M_K=M - \nu(K)$.  Second, we perform a \emph{Dehn filling} of $\partial\nu(K)$, which is the process of gluing a solid torus $\ST$ to the toroidal boundary component $\partial\nu(K) \subset \partial M_K$ via a homeomorphism $\phi$ which sends the meridian of $V$ to a simple closed curve $\gamma$ on the boundary of the knot exterior.  
    $$ \phi:\partial\ST \to \partial\nu(K) \subset \partial M_K, ~~ \phi( \partial D^2 \times \{0\})=\gamma$$
    We denote the result of $\gamma$ surgery on $K$ by $M_\gamma(K)$.
    Note that the knot exterior has one more torus boundary component than $M$ which corresponds to $\partial\nu(K)$. We can specify a basis of $H_1(\partial\nu(K))$ as $\mu_K$, the curve on $\partial\nu(K)$ which bounds a disk in $\nu(K)$, and a framing $\ell$. Consequently, every essential simple closed curve $\gamma \subset \partial\nu(K) $ is of the form $p\mu_K+q\ell$. Then, when a basis has been specified, we alternatively denote $M_\gamma(K)$ by $M_{p/q}(K)$ and refer to it as the $p/q$-Dehn surgery of $K$. Before we continue we note a few important facts and notations regarding Dehn surgery. The resultant $M_{p/q}(K)$ only depends on $p/q$ as an element of $\mathbb{Q} \cup \{\infty\}$, where $p/0 := \infty$. For null-homologous knots $K$ there is a canonical choice of framing $\lambda_K$ called the \emph{Seifert framing}. The result of $1/0$-Dehn surgery on any $K$ always returns the starting $3$-manifold $M$, i.e. $M_{1/0}(K)=M$. A framed link $(L,\ell)$ is link $L=L_1 \cup L_2 \cup \dots L_n$ for which every constituent knot $L_i$ is equipped with a framing $\ell_i$. We can specify a Dehn surgery on $L$ by choosing surgery coefficients $p_i/q_i$ for each framed knot and doing each surgery individually. For framed links we specify the coefficient $\ast$ to mean the corresponding component remains unfilled, e.g. 
    $$ M_{p_1/q_1,\ast}(L_1 \cup L_2)= M_{p_1/q_1}(L_1) - \nu(L_2).$$
    
    \par As mentioned previously we will use a satellite construction to find new representatives for $h$ using $K$. For the following definitions we will work with an arbitrary framed knot $K \subset M$.  The satellite operation takes in $(K,\ell)$ and a link $P$ in $V$ and outputs a new knot $P(K_\ell)$ in $M$. Given the additional information of a framing of $P$ we can construct a framing for $P(K_\ell)$.
\begin{Definition} \label{defSat}
    A satellite link $P(C_{\ell})$ with pattern link $P \subset \ST$, and companion knot $C \subset M$ is the image of $P$ under the oriented homeomorphism $\psi: \ST \to \nu(C)$ which sends the longitude of the solid torus to the framing $\ell$ of the knot $C$.
\end{Definition}
\par Note that when $M=S^3$ and $\ell=\lambda_C$ this is the standard satellite operation. We are particularly interested in the case when $P\subset V$ has components which are homotopic to a core curve and the \emph{pattern space} $V_P$ is a hyperbolic manifold.  
\begin{Definition} \label{SatFra}
    For any satellite knot $P(C_\ell) \subset M$ and any framing $\vartheta$ of $P$, there is an induced framing of $P(C_\ell)$ by $\psi(\vartheta)$.  
\end{Definition}
    We now describe the key scenario from which our framed pattern knots will arise. Suppose $L \subset S^3$ is a link in $S^3$ with a component $U$ which is unknotted. The exterior of $U$ is $V$ and the remaining components of $L$ can be thought of as embedded knots in $V$. Suppose $\tilde{P} \subset L$ is another component of $L$ and denote the corresponding knot in $S^3_U$ by $P$. Since $\tilde{P}$ is null-homologous in $S^3$ it has Seifert framing $\lambda_{\tilde{P}}$, the inclusion of this Seifert framing into $S^3_U$ is then a framing for the knot $P \subset V$. 

Our main tool for studying Dehn surgeries will be the JSJ decomposition. We note this outline closely follows that of Sorya \cite{Sorya}. Recall that for any compact irreducible orientable 3-manifold $M$, there is a minimal collection $\JSJt$ of properly embedded disjoint incompressible tori and annuli such that each component of $M \setminus \JSJt$ is either a hyperbolic or a Seifert fibred manifold, and this collection is unique up to isotopy. The \textit{JSJ decomposition} of $M$ is given by:
\[M = M_0 \cup M_1 \cup \ldots \cup M_k,\]
where each $M_i$ is the closure of a component of $M \setminus \JSJt$. A manifold $M_i$ is called a \emph{JSJ piece} of $M$ and a torus in the collection $\JSJt$ is called a \textit{JSJ torus} of $M$. We denote the collection of \emph{JSJ pieces} of $M$ by $JSJ(M)$. Any homeomorphism between compact irreducible orientable 3-manifolds can be seen as sending JSJ pieces to JSJ pieces, up to isotopy. 
\par We will often deal with the case that our particular $3$-manifold of interest $M$ is formed by gluing two hyperbolic $3$-manifolds $M_1$ and $M_2$ along incompressible boundary tori. In this case the resulting $3$-manifold must be irreducible because any newly formed essential $2$-sphere in $M$ can be isotoped into either $M_1$ or $M_2$ which is impossible since both are hyperbolic. The most common way we will see this is satellite knot exteriors. For example, given knots $K \subset S^3$ and $P \subset V$  with hyperbolic exteriors $M_K$ and $V_P$,  the exterior of the satellite knot $P(K)$ is irreducible and the JSJ decomposition is $V_P \cup M_K$. See Budney \cite{Bud} for a detailed exposition on JSJ decompositions of link exteriors in $S^3$. 

\section{Describing the knots $A_n$ and $B_n$}
In order to use and verify the key features of the links we construct we will need to use multiple different perspectives. First, we consider a four component link $\overline{L} \subset S^3$, as seen in Figure \ref{linkP}. We use the computer program SnapPy to analyze it. Second, we use $\overline{L}$ to construct a corresponding three component link $L=\gamma \cup \tilde{A}\cup B $ in $S^3 -\nu(\mu) \cong V$. Third, we perform $-1/n$-Dehn surgery of $\gamma$ which adds $n$ twists into a component of $L$ resulting in the two component link $L_n=\tilde{A_n}\cup B$ in $V$. Finally, $L_n$ is used to construct $A_n$ and $B_n$ in Definition \ref{patternref}.  Which will be the pattern knots used to construct new representatives of $h$.

\begin{figure}[h]
    \centering
    \labellist 
    \small \hair 2pt
    \pinlabel \textcolor{green}{$B$} at 200 400
    \pinlabel \textcolor{red}{$\tilde{A}$} at 570 300
    \pinlabel \textcolor{blue}{$\gamma$} at 320 200
    \pinlabel \textcolor{cyan}{$\mu$} at 190 275
    \endlabellist
    \includegraphics[width=3in]{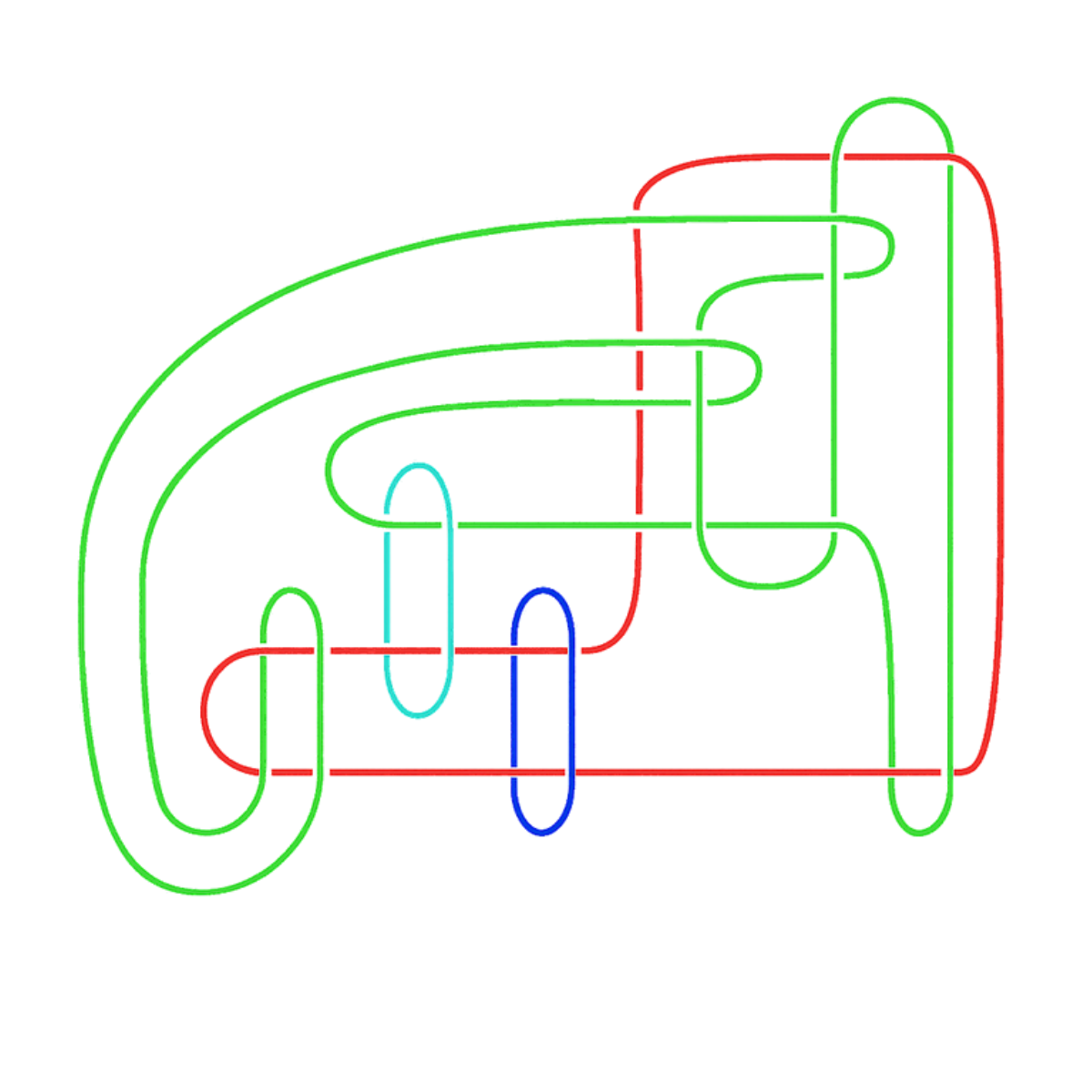}
    \caption{The link $\overline{L}$ in $S^3$}
    \label{linkP}
\end{figure}

\begin{Definition} \label{DefP}
    The link $\overline{L} \subset S^3$ is the four component link $\overline{L}=\gamma \cup \tilde{A} \cup B  \cup \mu$ as in Figure $\ref{linkP}$.  Similarly, we consider the links $L_n \subset V$. To construct $L_n$ from $L$ we do  $-\frac{1}{n}$-Dehn surgery of $\gamma$ using the Seifert framing from $S^3$ which adds $n$ horizontal twists into the knot $\tilde{A}$. We denote the $n$-twisted version of $\tilde{A}$ by $\tilde{A_n}$. 
\end{Definition}
    We now construct a set of manifolds by surgery on $L_n$. Note that surgeries on $L_n$ have an equivalent description by the $3$-component link $L \subset V$. 
\begin{figure}[h]
    \centering
    \labellist 
    \small \hair 2pt
    \pinlabel \textcolor{green}{$0$} at 200 400
    \pinlabel \textcolor{red}{$0$} at 570 300
    \pinlabel \textcolor{cyan}{$\ast$} at 190 275
    \pinlabel \textcolor{blue}{$-\frac{1}{n}$} at 330 200
    \endlabellist
    \includegraphics[width=3in]{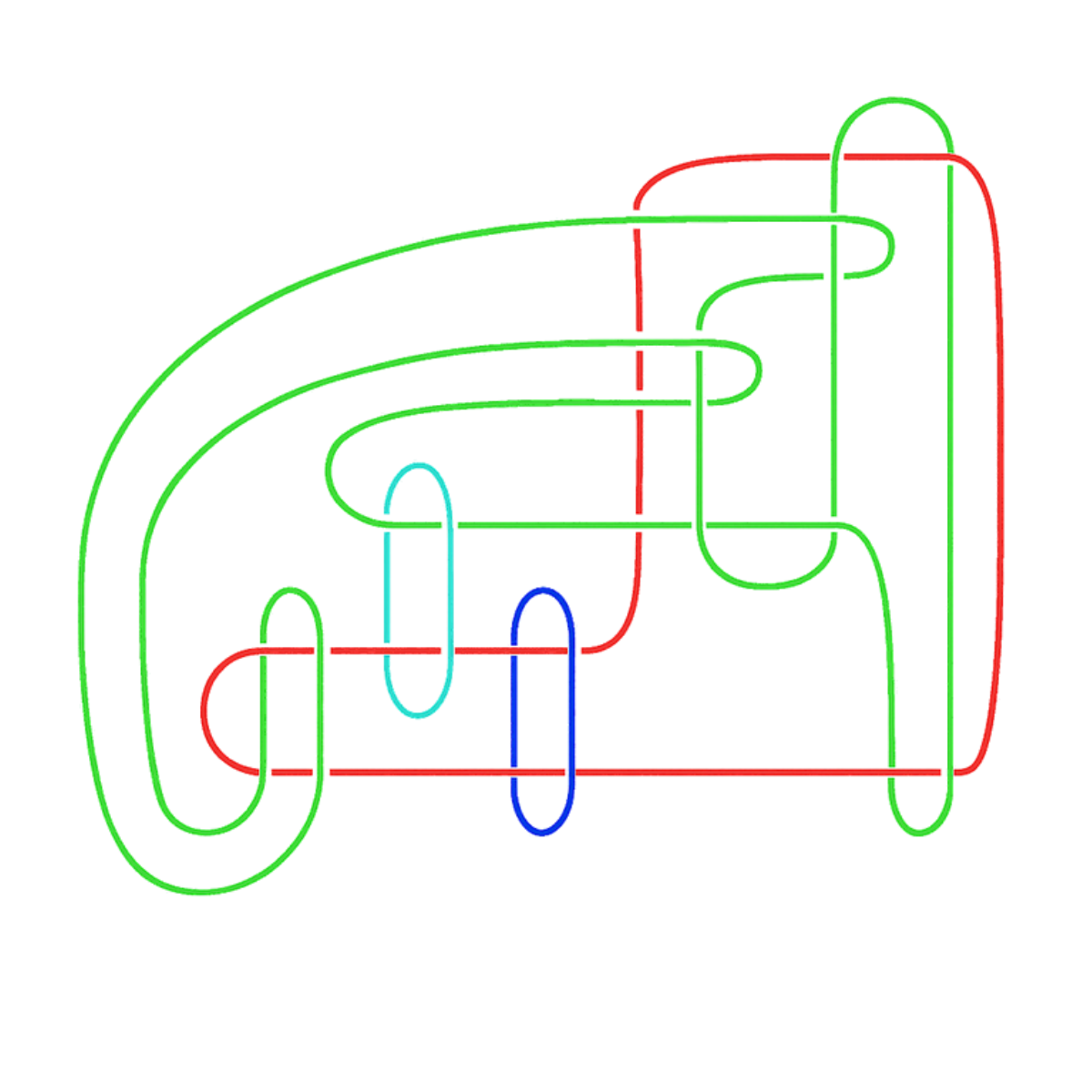}
    \caption{Surgery Diagram for $Z_n$}
    \label{Zee}
\end{figure}

\begin{Definition} \label{ManiLn}
     Fix $L_n \subset \ST$ as in Definition \ref{DefP}. We define the following sequences of manifolds for each $n \in \mathbb{N}$. 
\begin{enumerate}
    \item $\STmLone_n =\ST_{0, \ast}(\Slink)=\ST_{-1/n,0,\ast}(L)$ 
    \item $\STmLtwo_n=\ST_{\ast,0}(\Slink)=\ST_{-1/n,\ast,0}(L)$
    \item $\STmLonetwo_n=\ST_{0,0}(\Slink)=\ST_{-1/n,0,0}(L)$
\end{enumerate}
\end{Definition} 
To prove theorems about our sequences of manifolds we define an additional set of related manifolds which come from leaving $\gamma$ unfilled. 
\begin{Definition} \label{ManiL}
    Fix $L \subset \ST$ as in Definition \ref{DefP}. We define the following manifolds: $\Stml=\ST_{\ast,\ast,\ast}(L)$, $\STmLone=\ST_{\ast,0,\ast}(L)$, $\STmLtwo=\ST_{\ast,\ast,0}(L)$, and $\STmLonetwo=\ST_{(\ast,0,0)}(L)$.
\end{Definition}
\begin{Lemma} \label{listoflimits}
   Each of the manifolds from Definition $\ref{ManiL}$ is hyperbolic.  Additionally, $vol(\STmLone)=29.6209311377130\dots$ and $vol(\STmLtwo)=30.3314052251137\dots$.
\end{Lemma}
\begin{proof} \label{refhyper}
    Each manifold is constructed using a link diagram in the computer program SnapPy \cite{SnapPy} within Sage and is verified to be genuinely hyperbolic.
\end{proof}
In order to make use of the information from Lemma $\ref{listoflimits}$ to study the sequences from Definition \ref{ManiLn} we employ Thurston's hyperbolic Dehn surgery Theorem which relates a hyperbolic $3$-manifold $M$ with torus boundary to sequences of Dehn fillings. We do not state the theorem in full generality. Instead we state a straightforward corollary which succinctly captures which aspects of the theorem we need.  For any $M$ and any torus boundary component $T$, we can specify a Dehn filling of $M$ by a choice of simple closed curve $\gamma$ on $T$ to which the meridian of $V$ gets glued. We denote the filled manifold $M(\gamma)$.
\begin{Theorem} \label{thurston} \cite{Thurston} Let $M$ be a cusped hyperbolic $3$-manifold with a fixed cusp $T$. Fix a basis $\alpha,\beta$ for $H_1(T)$ and let ${p_i}/{q_i}$ be a sequence such that $p_i^2+q_i^2 \to \infty$.  Then $M(p_i\alpha+q_i\beta)$, is hyperbolic for all but finitely many $i$, furthermore $vol(M(p_i\alpha+q_i\beta))\nearrow vol(M)$. We note this sequence converges monotonically from below for sufficiently large $i$.
\end{Theorem} 
     Both $X$ and $Y$ have a boundary component corresponding to the knot $\gamma$ and are hyperbolic. Using Thurston's theorem we know that filling the cusp corresponding to $\gamma$ will produce another hyperbolic manifold except for finitely many exceptional slopes. By fixing the slope $-1/n$ and a large enough $n$ so that $-1/n$ is never exceptional we can then compare $X_n$ and $Y_n$. 
\begin{Lemma} \label{juicyLemma}
    For sufficiently large $n$ the manifolds $\STmLone_n$ and $\STmLtwo_n,$ are hyperbolic and pairwise non-homeomorphic. 
\end{Lemma}
\begin{proof}
    From Lemma \ref{listoflimits}  $\STmLone$ and $\STmLtwo$ are hyperbolic and have pairwise distinct volumes. Additionally, each has a boundary component corresponding to the knot $\gamma$ as seen in Definition \ref{ManiL}. We defined $\STmLone_n$ and $\STmLtwo_n,$ to be the result of filling this cusp by $-\frac{1}{n}$. From the assumption that $n$ is sufficiently large and Theorem \ref{refhyper}, $\STmLone_n$ and $\STmLtwo_n,$ are hyperbolic and their hyperbolic volumes approach  $vol(\STmLone)$ and $vol(\STmLtwo)$. Since $vol(\STmLtwo)>vol(\STmLone)$, once $n$ is sufficiently large $vol(\STmLtwo_n)>vol(\STmLone)$ and consequently $vol(\STmLtwo_n)>vol(\STmLone_n)$. Therefore, $\STmLone_n$ and $\STmLtwo_n$ have pairwise distinct volumes and are pairwise non-homeomorphic. 
\end{proof}
We now construct the knots $A_n$ and $B_n$ in $V$. Note both components of $L_n$ are isotopic to core curves, by performing $0$-surgery on one we again obtain $V$. Keeping track of the other component after surgery will give us a hyperbolic knot in $V$ which we will show is homotopic to a core curve.
\begin{figure}[h]
    \centering
    \labellist 
    \small \hair 2pt
    \pinlabel \textcolor{green}{$0$} at 200 400
    \pinlabel \textcolor{red}{$A_n$} at 570 300
    \pinlabel \textcolor{cyan}{$\ast$} at 190 275
    \pinlabel \textcolor{blue}{$-\frac{1}{n}$} at 330 200
    \endlabellist
    \includegraphics[width=3in]{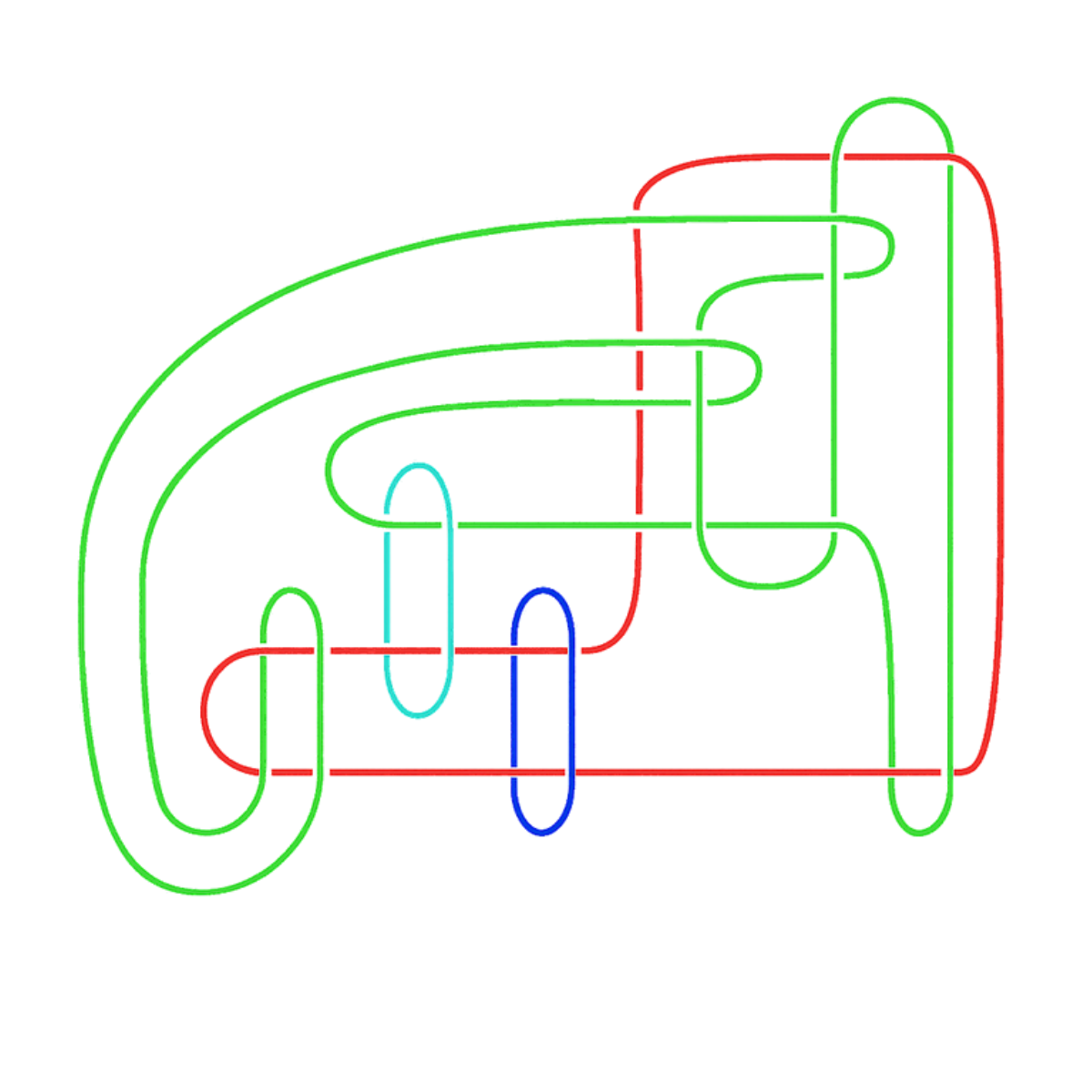}
    \caption{Diagram of $A_n$}
    \label{tildeAasknot}
\end{figure}
\begin{Definition} \label{patternref}
    The knot $A_n \subset V$ is the result of $\tilde{A_n}$ after performing zero surgery on $B$. This is depicted in Figure \ref{tildeAasknot}. The knot $B_n \subset V$ is the result of $B$ after performing zero surgery on $\tilde{A_n}$.  
\end{Definition}
\par All of our work in Section $3$ so far has been to produce an infinite family of pairs of knots in $V$. These pairs must satisfy the following key properties: there are infinitely many of them which are hyperbolic and distinct as pairs. Additionally, the knots making up each pair must have non-homeomorphic exteriors, homemorphic 0 surgeries, and be homotopic to core curves. We note the exteriors of $A_n$ and $B_n$ exactly correspond to $Y_n$ and $X_n$ by definition, which were shown to be distinct for sufficiently large $n$ in Lemma $\ref{juicyLemma}$. Furthermore, both their $0$-Dehn surgeries result in $Z_n$ because in either case we have done $(0,0)$ surgery on $L_n$. We note the pairs are distinct by the following two observations. First, in Lemma $\ref{juicyLemma}$ we showed that $vol(Y_n)>vol(X)>vol(X_m)$ which implies $A_n$ and $B_m$ have non-homeomorphic exteriors for sufficiently large $m,n$. Second, we note there are no infinite subsequence of the $X_n$'s with fixed volume. Suppose one did exist, then each element in the sequence would have to have the same volume as $X$. However, this is impossible because Dehn filling always decreases the volume of the resulting manifold. As such every $A_n$ and $A_m$ have different volumes for $n,m$ sufficiently large. We now prove a short lemma that $A_n$ and $B_n$ are homotopic which in part relies on the the fact we have a surgery diagram for these knots in $S^3$ so linking numbers are well defined. 
\begin{Lemma} \label{homotopyLemma}
    The knots $A_n$ and $B_n$ are homotopic to core curves in $V$. 
\end{Lemma}
\begin{proof}
    $A_n$:  When we remove a neighborhood of $B$, because it is isotopic to a core curve, we obtain $V-\nu(B) \cong T^2 \times I$. Recall a framing of $B$ and the meridian of $B$ define a basis for $\pi_1(\partial \nu(B))$. Now we choose our basis for $\pi_1(\partial \nu(B))$ as the zero framing $\lambda_B$ and $\mu_B$. Since $V_B \cong T^2 \times I$, it is homotopy equivalent to $\partial\nu(B)$ and $(\mu_B,\lambda_B)$ give a basis for $\pi_1(V_B)$. In this basis we know $[A_n]=lk(B,\tilde{A_n})[\mu_B]+lk(\mu,\tilde{A_n})[\lambda_B]=[\mu_B]+[\lambda_B]$ for any $n$. Now in $V_0(B) \cong V$ we attached a disk exactly to $\lambda_B$ which kills this generator, as such $[\tilde{A}_n]$ represents a positive generator of $\pi_1(V)$ so it is homotopic to a core curve. 
\\
    $B_n$: The argument for $B_n$ is given by switching the roles of $B$ and $\tilde{A_n}$.
    \end{proof}

\section{Main Theorem}
    Before we begin the proof we note that all of the manifolds which appear in the main proof are of the form described at the end of Section $2$. Namely, they are formed by gluing two hyperbolic manifolds along an incompressible torus and are therefore irreducible. We also provide additional context for an assertion that appears in the main proof. When doing surgery on a core curve $c$ of $V$ in the construction of $A_n$ or $B_n$ we obtain a solid torus i.e. $V_0(c) \cong V$. However the surgery solid torus $V_0(c)$ undergoes a key change when compared to the original $V$. Namely, the meridional curves and longitudinal curves of $V$ become  longitudinal and meridional curves of $V_0(c)$.  Recall, in the construction of the knots $A_n$ and $B_n$ we did surgery on a core curve of $V$. For any knot $K \subset M$ we have two alternative constructions of the satellite knots $A_n(K_\ell)$ and $B_n(K_\ell)$. First, construct $A_n$ or $B_n$ as a knot in $V_0(c)$ then do a satellite. Second, treat $P_n$ as a knot in $V$ and form the satellite link $P_n(K_\ell)$ then do $0$-surgery of the appropriate component of $P_n(K)$ to obtain $A_n(K_\ell)$ or $B_n(K_\ell)$. The first version distinguishes the knots $A_n(K_\ell)$ or $B_n(K_\ell)$ more easily and the second makes the equivalence of their $0$-surgeries more clear.

\begin{proof} [Proof of Theorem \ref{mainthm}]
    By a Theorem of Myers \cite{Myers}, since $\partial M$ contains no $2$-spheres, every homotopy class $h$ in $\pi_1(M)$ has a knot representative $K$ which is hyperbolic. Choose a framing $\ell$ for $K$ and an $n$ sufficiently large for Lemma $\ref{juicyLemma}$ to apply. Consider the satellite knots $A_n(K_\ell)$ and $B_n(K_\ell)$ using the pattern knots from Definition $\ref{patternref}$. We know the knot exteriors $M - \nu ({A_n(K_\ell)})$ and $M - \nu ({B_n}(K_\ell))$ have JSJ decompositions with two pieces, the pattern space and $M_K$. Therefore, the JSJ decompositions are  $M_K \cup Y_n$ and $M_K \cup X_n$ respectively. Any homeomorphism of $3$-manifolds must send JSJ pieces to JSJ pieces and the manifolds $X_n$ and $Y_n$ are non-homeomorphic by Lemma $\ref{juicyLemma}$, therefore $M - \nu ({A_n(K_\ell)})$ and $M - \nu ({B_n}(K_\ell))$ are distinct. In order to show $M_0(A_n(K_\ell)) \cong M_0(B_n(K_\ell)) $ we first construct the satellite link $L_n(K_\ell)$ using the link $L_n$ from definition $\ref{DefP}$. As described at the beginning of section $4$, doing $0$-surgery of one component of $L_n(K_\ell)$ and leaving the other component unfilled corresponds to the exteriors of $A_n(K_\ell)$ and $B_n(K_\ell)$ so, $M_{*,0}(L_n(K_\ell)) \cong M - \nu ({A_n(K_\ell)})$ and $M_{0,*}(L_n(K_\ell)) \cong M - \nu ({A_n(K_\ell)})$. Then $0$-surgery of $A_n$ or $B_n$ corresponds exactly with $(0,0)$ surgery of $P_n(K_\ell)$ and as a result $M_{0}(A_n(K_\ell)) \cong M_{0,0}(L_n(K_\ell)) \cong M_{0}(B_n(K_\ell))$. Finally, from the fact that $A_n$ and $B_n$ are homotopic to core curves of $V$ as described in Lemma $\ref{homotopyLemma}$, the knots $A_n(K_\ell)$ and $B_n(K_\ell)$ are both homotopic to the core curve of $\nu(K)$ in $M$ which is exactly $K$. 
\end{proof}

\bibliography{ref2}{}
\bibliographystyle{plainurl}

\end{document}